\documentclass[a4paper,11pt,reqno]{amsart}
\usepackage{amssymb}
\usepackage{amsmath,enumerate}
\usepackage{graphicx}
\usepackage{tikz}

\usepackage[top=30truemm,bottom=30truemm,left=20truemm,right=20truemm]{geometry}
\allowdisplaybreaks[4]

\newtheorem{Def}{Definition}[section]
\theoremstyle{remark}
\newtheorem{Rem}[Def]{Remark}

\newtheorem*{Ack}{Acknowledgments}
\theoremstyle{plain}
\newtheorem{Th}[Def]{Theorem}
\newtheorem{Prop}[Def]{Proposition}
\newtheorem{Lem}[Def]{Lemma}
\newtheorem{Cor}[Def]{Corollary}
\newtheorem{Fact}[Def]{Fact}

\newcommand{\Z}{\mathbb{Z}}

\newcommand{\C}{\mathbb{C}}
\renewcommand{\P}{\mathbb{P}}
\renewcommand{\H}{\mathbb{H}}

\newcommand{\CL}{\mathcal{L}}

\newcommand{\CO}{\mathcal{O}}

\newcommand{\al}{\alpha }
\newcommand{\be}{\beta }
\newcommand{\ga}{\gamma }

\newcommand{\Ga}{\Gamma }

\newcommand{\vth}{\vartheta }

\newcommand{\vph}{\varphi }
\newcommand{\la}{\lambda }
\newcommand{\La}{\Lambda }
\newcommand{\om}{\omega }
\newcommand{\Om}{\Omega }
\newcommand{\si}{\sigma }
\newcommand{\na}{\nabla }

\newcommand{\op}{\oplus }

\newcommand{\bu}{\bullet}

\newcommand{\tpi}{2\pi \sqrt{-1}}
\newcommand{\pii}{\pi \sqrt{-1}}

\newcommand{\bfe}{\mathbf{e}}

\DeclareMathOperator{\reg}{reg}

\DeclareMathOperator{\cn}{cn}
\DeclareMathOperator{\sn}{sn}
\DeclareMathOperator{\dn}{dn}
\DeclareMathOperator{\cs}{cs}
\DeclareMathOperator{\ds}{ds}
\DeclareMathOperator{\ns}{ns}
\newcommand{\h}{{\rm h}}
\newcommand{\ch}{{\rm ch}}
\newcommand{\lf}{{\rm lf}}
\newcommand{\cpt}{{\rm c}}

\def\tp#1{\mathord{\mathopen{{\vphantom{#1}}^t}#1}} 

\makeatletter
\@addtoreset{equation}{section}

\makeatother

\title[Twisted (co)homology for the Wirtinger integral]{
  Notes on twisted homology and cohomology groups for the Wirtinger integral
}
\author[Y. Goto]{Yoshiaki Goto}
\address[Goto]{
Otaru University of Commerce, 
3-5-21, Midori, Otaru, Hokkaido, 047-8501, JAPAN
}
\email{goto@res.otaru-uc.ac.jp}
\author[G. Shibukawa]{Genki Shibukawa}
\address[Shibukawa]{
Kitami Institute of Technology,
165, Koen-cho, Kitami, Hokkaido, 090-8507, JAPAN
}
\email{g-shibukawa@mail.kitami-it.ac.jp}

\date{\today}

\begin{document}
\begin{abstract}
  The Wirtinger integral is one of the integral representations of the Gauss hypergeometric function. 
  Its integrand is given by a product of complex powers of theta functions. 
  We study the structure of the twisted homology and cohomology groups associated with this integral. 
  Using the involution on the complex torus, 
  we show that these groups decompose into eigenspaces 
  which are orthogonal with respect to the intersection forms. 
  Each eigenspace is related to the twisted (co)homology group 
  associated with the Euler-type integral representation of the Gauss hypergeometric function. 
  We also show that the corresponding intersection matrices admit simple forms. 
\end{abstract}

\maketitle

\section{Introduction}
The Gauss hypergeometric function ${}_2 F_1 (\al,\be,\ga;z)$ admits an Euler-type integral representation
\begin{align}
  \label{eq:gauss-integral}
  {}_2 F_1 (\al,\be,\ga;z)=\frac{\Ga(\ga)}{\Ga(\al) \Ga(\ga-\al)} 
  \int_0^1 t^{\al} (1-t)^{\ga-\al} (1-zt)^{-\be} \frac{dt}{t(1-t)} .
\end{align}
If we set $z=\la(\tau)$ ($\tau \in \H$), ${}_2 F_1 (\al,\be,\ga;\la(\tau))$ 
also admits the representation
\begin{align*}
  \int_0^{\frac{1}{2}} \vth_1(u)^{2\al-1} \vth_2(u)^{2\ga-2\al-1} \vth_3(u)^{-2\be+1} \vth_4(u)^{2\be -2\ga +1} du,
\end{align*}
see \S\ref{subsec:theta-function} for details. 
Such an integral is called the Wirtinger integral. 
The integrand naturally defines a local system $\CL$ and its dual $\CL^{\vee}$ on $M$, where $M$ is 
a complex torus with four $2$-torsion points removed. 
The Wirtinger integral is obtained as a pairing between 
the twisted cohomology group $H^1(M;\CL)$ and the twisted homology group $H_1(M;\CL^{\vee})$. 
These groups are studied in 
\cite{Watanabe-elliptic-homology-cohomology}, \cite{Watanabe-wirtinger-diff-eq},  
and their intersection theory is studied in \cite{Ghazouani-Pirio}, \cite{G-RW-intersection}. 
In this paper, we give a more detailed study of these groups and the intersection forms. 
We focus on the involution $\iota :M\to M$ defined by $\iota(u)=-u$. 
This involution induces actions on $H^1(M;\CL)$ and $H_1(M;\CL^{\vee})$, 
which yield eigenspace decompositions 
\begin{align*}
  H^1(M;\CL) = H^{(-1)} \op H^{(1)} ,\quad 
  H_1(M;\CL^{\vee})=H_{(-1)}^{\vee} \op H_{(1)}^{\vee} .
\end{align*}
Although such a decomposition of the twisted (co)homology groups is mentioned in 
\cite[Example 3.1]{Mano-generalization}, 
we show more precisely that
these decompositions are in fact ``orthogonal'' with respect to the intersection forms; 
see Proposition~\ref{prop:cohomology-decomp} and Corollary~\ref{cor:homology-decomp} for details. 
Furthermore, we explicitly construct bases of these eigenspaces 
and compute the corresponding intersection matrices. 

This paper is organized as follows. 
In Section~\ref{sec:preliminaries}, 
we review basic properties of theta functions, 
and results of \cite{Mano-Watanabe}, \cite{Watanabe-elliptic-homology-cohomology} and \cite{Watanabe-wirtinger-diff-eq} 
on the twisted homology and cohomology groups. 
Sections~\ref{sec:cohomology} and \ref{sec:homology} form the central part of this paper. 
We study the involution $\iota$ and the eigenspace decompositions of the twisted homology and cohomology groups. 
In Section~\ref{sec:cohomology}, 
we compute the intersection matrix for a basis $[\phi_1],\dots ,[\phi_4] \in H^1(M;\CL)$ 
constructed following the ideas of \cite{Watanabe-elliptic-homology-cohomology} and \cite{Watanabe-wirtinger-diff-eq}. 
In fact, $[\phi_1],\dots ,[\phi_4]$ form bases of the eigenspaces $H^{(\pm 1)}$. 
The intersection number $\langle [\phi_2],[\phi_2] \rangle_{\ch}$ has a complicated expression 
in terms of theta functions, 
which we reduce to a simple form using elliptic functions. 
In Section~\ref{sec:homology}, 
we first compute the intersection matrix for a basis $[\si_1],\dots ,[\si_4] \in H_1(M;\CL^{\vee})$ 
consisting of twisted cycles similar to those given in \cite{Watanabe-wirtinger-diff-eq}, 
and show that the integrals over these cycles can be expressed in terms of ${}_2 F_1$. 
However, these cycles do not belong to the eigenspaces $H_{(\pm 1)}^{\vee}$. 
We then give an explicit construction of bases of these eigenspaces. 
In Section~\ref{sec:TPR}, 
we obtain twisted period relations 
as a corollary of the explicit forms of the intersection matrices. 
Using our bases of the eigenspaces, these relations are reduced to simple forms 
and can essentially be derived from the twisted (co)homology theory associated with 
the Euler-type integral~(\ref{eq:gauss-integral}).
In Appendix~\ref{sec:direct-proof}, we verify this fact.

\section{Preliminaries}\label{sec:preliminaries}
In this section, we review basic properties of theta functions and of twisted homology and cohomology groups,
which will be used throughout this paper. 

\subsection{Theta function}\label{subsec:theta-function}
For $u\in \C$, we set $\bfe(u)=e^{\tpi u}$. 
We define theta functions 
\begin{align*}
  \vth_1 (u)=\vth_1 (u,\tau)
  &=-\sum_{m\in \Z} \bfe\left(\frac{1}{2}\Big( m+\frac{1}{2} \Big)^2 \tau
  +\Big( m+\frac{1}{2} \Big)\Big( u+\frac{1}{2} \Big)\right), \\
  \vth_2 (u)=\vth_2 (u,\tau)
  &=\sum_{m\in \Z} \bfe\left( \frac{1}{2} \Big( m+\frac{1}{2} \Big)^2 \tau
  +\Big( m+\frac{1}{2} \Big) u \right), \\
  \vth_3 (u)=\vth_3 (u,\tau)
  &=\sum_{m\in \Z} \bfe\left(\frac{1}{2} m^2 \tau
  + m u\right), \\
  \vth_4 (u)=\vth_4 (u,\tau)
  &=\sum_{m\in \Z} \bfe \left( \frac{1}{2} m^2 \tau
  +m\Big( u+\frac{1}{2} \Big)\right), 
\end{align*}
where $u\in \C$ and $\tau \in \H$. 
In this paper, we fix $\tau \in \H$ and frequently use the notation $\vth_j (u)$. 
We note that $\vth_{1}(u,\tau)$, $\vth_{2}(u,\tau)$, $\vth_{3}(u,\tau)$ and $\vth_{4}(u,\tau)$
defined here are equal to $-\vth_{11}(u,\tau)$, $\vth_{10}(u,\tau)$, $\vth_{00}(u,\tau)$ and $\vth_{01}(u,\tau)$
in \cite{Mumford}, respectively. 
We also note that 
$\vth_{j}(u)$ ($j=2,3,4$) can be expressed by $\vth_{1}(u)$: 
\begin{align*}
  \vth_2(u)
  =-\vth_{1}\Big( u-\frac{1}{2} \Big), \quad
  \vth_3(u)
  =\bfe \Big( \frac{1}{2}+\frac{\tau}{8} -\frac{u}{2} \Big)\vth_{1} \Big( u-\frac{1+\tau}{2} \Big) ,\quad
  \vth_4(u)
  =\bfe \Big( \frac{1}{4}+\frac{\tau}{8} -\frac{u}{2} \Big) \vth_{1} \Big( u-\frac{\tau}{2} \Big) .
\end{align*}

Following to \cite{Watanabe-elliptic-homology-cohomology}, 
we substitute the integral in (\ref{eq:gauss-integral}) by 
\begin{align}
  \label{eq:henkan-sn}
  z=\la(\tau)=\frac{\vth_2^4(0)}{\vth_3^4(0)} ,\quad 
  t=\frac{\vth_3^2(0)}{\vth_2^2(0)}\frac{\vth_1^2(u)}{\vth_4^2(u)} .
\end{align}
Then, we have 
\begin{align}
  \nonumber
  &\int_{0}^{1} t^{\al-1} (1-t)^{\ga-\al-1} (1-\la(\tau)t) ^{-\be} dt \\
  \label{eq:W-integral-gauss}
  &=2\pi \vth_3(0)^2\la(\tau)^{\frac{1-\ga}{2}} (1-\la(\tau))^{\frac{\ga-\al-\be}{2}}
    \int_0^{\frac{1}{2}} \vth_1(u)^{2\al-1} \vth_2(u)^{2\ga-2\al-1} \vth_3(u)^{-2\be+1} \vth_4(u)^{2\be -2\ga +1} du .
\end{align}

\subsection{Twisted homology and cohomology groups}\label{subsec:homology-cohomology}
We set $\La_{\tau}=\Z +\Z \tau$, $E=\C/\La_{\tau}$,  
$D=\{ 0,\frac{1}{2},\frac{1+\tau}{2} ,\frac{\tau}{2} \}$ and $M=E-D$. 
We define a multi-valued function $T(u)$ on $M$ by 
\begin{align}
  \nonumber
  T(u)
  &= \vth_1(u)^{2\al} \vth_2(u)^{2\ga-2\al} \vth_3(u)^{-2\be} \vth_4(u)^{2\be -2\ga} \\
  \label{eq:T(u)-def-2}
  &= (\text{constant})\cdot e^{\tpi c_0 u}\vth_1(u-t_1)^{c_1}\vth_1(u-t_2)^{c_2}\vth_1(u-t_3)^{c_3}\vth_1(u-t_4)^{c_4} , 
\end{align}
where
\begin{align}
  \nonumber
  && &t_1=0,& &t_2=\frac{1}{2},& &t_3=\frac{1+\tau}{2},& &t_4=\frac{\tau}{2},& \\
  \label{eq:RW-setting}
  &c_0=\ga, & &c_1=2\al,& &c_2=2\ga-2\al,& &c_3=-2\be,& &c_4=2\be-2\ga.& 
\end{align}
Let $\CO_E(*D)$ (resp. $\Om^1_E(*D)$) denote the sheaf of functions (resp. $1$-forms) 
that are meromorphic on $E$ and holomorphic on $M$. 
Let $\CL$ and $\CL^{\vee}$ be the local systems on $M$ defined by $T(u)^{-1}$ and $T(u)$, respectively;  
$\CL=\C T(u)^{-1}$ and $\CL^{\vee}=\C T(u)$. 
We consider the homology groups $H_i(M;\CL^{\vee})$ and cohomology groups $H^i(M;\CL)$, 
called the twisted homology groups and twisted cohomology groups, respectively. 
Recall that the twisted homology group is defined by 
$H_i(M;\CL^{\vee})=Z_i(M;\CL^{\vee})/B_i(M;\CL^{\vee})$, where $Z_i$ and $B_i$ denote 
the kernel and image of the boundary operator on the twisted chains, respectively (e.g., \cite{AK}). 
We set $\om =d\log T(u)\in \Om^1_E(*D)(E)$. 

Using the expression (\ref{eq:T(u)-def-2}), we obtain basic facts for the twisted (co)homology groups 
by setting $\la=0$ and $n=4$ in the Riemann-Wirtinger integral \cite{Mano-Watanabe}. 
\begin{Fact}[\cite{Mano-Watanabe}, \cite{Watanabe-wirtinger-diff-eq}]\label{fact:vanish-dim}
  Suppose that $c_1,\dots ,c_4$ in (\ref{eq:RW-setting}) are not integers. 
  If $i\neq 1$, then $H_i(M;\CL^{\vee})=0$ and $H^i(M;\CL)=0$. We have 
  \begin{align*}
    \dim H_1(M;\CL^{\vee}) =\dim H^1(M;\CL) =4, \qquad
    H^1(M;\CL)\simeq \Om^1_E(*D)(E)/\na(\CO_E(*D)(E)) ,
  \end{align*}
  where $\na :\CO_E(*D) \to \Om^1_E(*D)$ is defined by $\na f =df +f\om$. 
\end{Fact}
Hereafter, we identify $H^1(M;\CL)$ with $\Om^1_E(*D)(E)/\na(\CO_E(*D)(E))$. 
We call an element of $Z_1(M;\CL^{\vee})$ (resp. $\Om^1_E(*D)(E)$) a twisted cycle (resp. cocycle). 
The natural pairing $H^1(M;\CL)\times H_1(M;\CL^{\vee}) \to \C$ given by 
$([\vph] ,[\si]) \mapsto \int_{\si} T(u)\vph$ is non-degenerate, 
which we call the period pairing. 
This pairing yields the Wirtinger integral.

We define twisted (co)cycles that form a basis of the twisted (co)homology group. 
Let $\si_1,\dots ,\si_4$ be twisted cycles defined as 
\begin{align*}
  \si_1 =\reg \Big(0,\frac{1}{2} \Big) , \quad 
  \si_2 =\reg \Big(\frac{1}{2},\frac{1+\tau}{2} \Big) ,\quad
  \si_3 =\reg \Big(\frac{1+\tau}{2},\frac{\tau}{2} \Big) ,\quad 
  \si_4 =\reg (0,1) ,
\end{align*}
where ``reg'' means the regularization of the locally finite cycle. 
Under the setting (\ref{eq:RW-setting}), we can see that 
$\si_1$, $\si_2$, $\si_3$ and $\si_4$ correspond to   
$\ga_{12}$, $\ga_{23}$, $\ga_{34}$ and $\ga_{10}$ in \cite{G-RW-intersection}, respectively. 
For precise definitions, see \cite{G-RW-intersection}, \cite{Mano-Watanabe}. 
We set 
\begin{align*}
  &\phi_1 =\pi \vth_{3}(0)^2 du ,  \quad
    \phi_2 =\pi \vth_{2}(0)^2 \frac{\vth_{4}(u)^2}{\vth_{1}(u)^2}du ,  \\
  &\phi_3 =\pi \vth_{2}(0)^2 \frac{\vth_{3}(u)\vth_{4}(u)}{\vth_{1}(u) \vth_{2}(u)}du
    =d\log \frac{\vth_{2}(u)}{\vth_{1}(u)} ,  \quad
  \phi_4 =\pi \vth_3(0)^2 \frac{\vth_2(u)\vth_4(u)}{\vth_1(u)\vth_3(u)}du
    =d\log \frac{\vth_{3}(u)}{\vth_{1}(u)}. 
\end{align*}
We note that $\phi_3$ and $\phi_4$ coincide with $-\om_1$ and $-\om_3$
in \cite{Watanabe-elliptic-homology-cohomology}, respectively. 
The image of $([\phi_3],[\si_1])$ under the period pairing agrees with  
the Wirtinger integral in (\ref{eq:W-integral-gauss}).

\begin{Fact}[\cite{Watanabe-elliptic-homology-cohomology},\cite{Watanabe-wirtinger-diff-eq}]
  \label{fact:homology-cohomology-basis}
  If $c_0,\dots ,c_4$ in (\ref{eq:RW-setting}) are not integers, 
  $\{ [\si_{j}] \}_{j=1,\dots ,4}$ and $\{ [\phi_{j}] \}_{j=1,\dots ,4}$
  form a basis of $H_1(M;\CL^{\vee})$ and $H^1(M;\CL)$, respectively. 
\end{Fact}
Here we assume the additional condition $c_0\notin \Z$, which is necessary to show 
that $\{ [\si_{j}] \}_{j=1,\dots ,4}$ form a basis. 
Although we can take another basis under the weaker assumption $c_1,\dots ,c_4\notin \Z$, 
we use $\{ [\si_{j}] \}_{j=1,\dots ,4}$ for computational simplicity.

\subsection{Intersection forms}
The intersection forms $\langle \bu ,\bu \rangle_{\h}$ and $\langle \bu ,\bu \rangle_{\ch}$ are non-degenerate bilinear forms
\begin{align*}
  \langle \bu ,\bu \rangle_{\h} : H_1(M;\CL^{\vee}) \times H_1(M;\CL) \longrightarrow \C ,\quad
  \langle \bu ,\bu \rangle_{\ch} : H^1(M;\CL) \times H^1(M;\CL^{\vee}) \longrightarrow \C .
\end{align*}
For a twisted cycle $\si \in Z_1(M;\CL^{\vee})$, we can construct $\si^{\vee} \in Z_1(M;\CL)$ by 
replacing $(\al,\be,\ga)$ with $(-\al,-\be,-\ga)$ 
(hence each $c_j$ is replaced by $-c_j$ ($j=0,\dots ,4$)). 
By Fact~\ref{fact:vanish-dim}, 
we have $H^1(M;\CL^{\vee})\simeq \Om^1_E(*D)(E)/\na^{\vee}(\CO_E(*D)(E))$,
where $\na^{\vee} f =df -f\om$. 
Thus, we also identify  
$H^1(M;\CL^{\vee})$ with $\Om^1_E(*D)(E)/\na^{\vee}(\CO_E(*D)(E))$. 
\begin{Rem}
  In general, 
  when we define the intersection forms and Poincar\'e duality (used later), 
  we need to consider 
  the locally finite twisted homology group $H_1^{\lf}(M;\CL^{\vee})$ and 
  the twisted cohomology group with compact support $H^1_{\cpt}(M;\CL)$. 
  Since there are natural isomorphisms $H_1(M;\CL^{\vee}) \simeq H_1^{\lf}(M;\CL^{\vee})$ and 
  $H^1_{\cpt}(M;\CL) \simeq H^1(M;\CL)$ (see, e.g., \cite[Theorem~3.4.4]{Dimca}), 
  we identify these groups and omit the labels ``lf'' and ``c''. 
\end{Rem}

\section{Structure of the twisted cohomology group}\label{sec:cohomology}
\subsection{Intersection matrix}
The cohomology intersection numbers can be computed in a manner similar to that in \cite{G-RW-intersection}, 
though we use a different basis.  
\begin{Th}\label{th:cohomology-intersection-matrix}
  The intersection matrix $C=(\langle [\phi_i],[\phi_j] \rangle_{\ch})_{i,j=1,\dots ,4}$
  is given by 
  \begin{align*}
    C=\tpi
    \begin{pmatrix}
      0 & \frac{1}{c_1+1} & 0 & 0 \\
      \frac{1}{c_1-1} & C_{22} & 0 & 0 \\
      0 & 0 & \frac{c_1 +c_2}{c_1 c_2} &\frac{1}{c_1} \\
      0 & 0 & \frac{1}{c_1} & \frac{c_1 +c_3}{c_1 c_3} 
    \end{pmatrix},
  \end{align*}
  where
  \begin{align}
    \label{eq:C22-1}
    C_{22}
    &=\frac{1}{\pi^2 \vth_{3}(0)^4(c_1 -1)(c_1 +1)}
    \Big(-c_1 \frac{\vth_{1}'''(0)}{\vth_{1}'(0)}-c_2\frac{\vth''_{2}(0)}{\vth_{2}(0)} 
    -c_3\frac{\vth''_{3}(0)}{\vth_{3}(0)} +(2c_1-c_4)\frac{\vth''_{4}(0)}{\vth_{4}(0)}\Big) \\
    \label{eq:C22-2}
    &=\frac{(a-b+1)\la(\tau)+c}{2a(a+1)}.
  \end{align}
  Here we set $\al=a+\frac{1}{2}$, $\be=b-\frac{1}{2}$, $\ga=c$.
\end{Th}
\begin{proof}
  For a detailed method to obtain the intersection numbers, 
  see \cite[\S 5.2]{G-RW-intersection}. 
  For the reader's convenience, we write down the first few terms of the Laurent expansions of $\phi_2$ and $f$ around $u=0$, 
  where $f$ satisfies $\na f =\phi_2$: 
  \begin{align*}
    \frac{\phi_2}{du} &=\frac{1}{\pi \vth_{3}^2}\cdot \frac{1}{u^2}
    +\frac{1}{\pi \vth_{3}^2} \left( \frac{\vth_{4}''}{\vth_{4}}-\frac{\vth_{1}'''}{3\vth_{1}'} \right) 
    +\cdots , \\
    f&=\frac{1}{c_1-1}\frac{1}{\pi \vth_{3}^2}\cdot \frac{1}{u}
       +\frac{1}{(c_1-1)(c_1+1)} \frac{1}{\pi \vth_{3}^2} \left(
       \frac{-2c_1+1}{3}\frac{\vth'''_{1}}{\vth'_{1}} -c_2\frac{\vth''_{2}}{\vth_{2}}
       -c_3\frac{\vth''_{3}}{\vth_{3}}+(c_1-c_4-1)\frac{\vth''_{4}}{\vth_{4}} 
       \right) u +\cdots .
  \end{align*}
  Here, we write $\vth'_1=\vth'_1(0)$, $\vth_2=\vth_2(0)$, and so on. 
  By this computation, we obtain the expression (\ref{eq:C22-1}). 
  Since the proof of (\ref{eq:C22-2}) is slightly long, we prove it in \S\ref{subsec:proof-of-C22}. 
\end{proof}

\subsection{Involution and decomposition}\label{subsec:involution-cohomology}
Let $\iota :E \to E$ be the involution defined by $\iota(u)=-u$. 
Since $\iota (D)=D$, 
it also induces an involution on $M$. 
By the relation $\iota^{*}\om =\om$ and Theorem~\ref{th:cohomology-intersection-matrix}, 
we obtain the following proposition. 

\begin{Prop}\label{prop:cohomology-decomp}
  The involution $\iota^{*} :\Om^1_E(*D)(E) \to \Om^1_E(*D)(E)$ induces 
  automorphisms $\iota^{*} :H^1(M;\CL)\to H^1(M;\CL)$ and 
  $\iota^{*\vee} :H^1(M;\CL^{\vee})\to H^1(M;\CL^{\vee})$. 
  These cohomology groups decompose into the eigenspaces: 
  \begin{align*}
    H^1(M;\CL) = H^{(-1)} \op H^{(1)} ,\quad 
    H^1(M;\CL^{\vee}) = H^{(-1)\vee} \op H^{(1)\vee}, 
  \end{align*}
  where $H^{(\pm 1)}=\{ [\vph] \in H^1(M;\CL) \mid \iota^{*}[\vph]=\pm [\vph] \}$, and 
  similarly for $H^{(\pm 1)\vee}$. 
  The cohomology classes $[\phi_1]$ and $[\phi_2]$ (resp. $[\phi_3]$ and $[\phi_4]$) form bases of 
  $H^{(-1)}$ and $H^{(-1)\vee}$ (resp. $H^{(1)}$ and $H^{(1)\vee}$). 
  Further, $H^{(\pm 1)}$ and $H^{(\mp 1)\vee}$ are orthogonal to each other with respect to 
  the intersection form $\langle \bu,\bu \rangle_{\ch}$. 
\end{Prop}
\begin{Rem}
  Although the orthogonality of the eigenspaces is verified 
  by Theorem~\ref{th:cohomology-intersection-matrix}, 
  it also follows from the property 
  $\langle [\vph],[\psi] \rangle_{\ch}=\langle \iota^{*}[\vph],\iota^{*\vee}[\psi]\rangle_{\ch}$. 
\end{Rem}

\begin{Rem}\label{rem:double-cover}
  We set $X=\P_t^1 -\{ 0,1,1/z,\infty \}$ and $U(t)=t^{\al} (1-t)^{\ga-\al} (1-zt)^{-\be}$. 
  For a local system $\CL_0 =\C U(t)^{-1}$, we consider the twisted cohomology group $H^1(X;\CL_0)$
  associated with the integral (\ref{eq:gauss-integral}). 
  The correspondence (\ref{eq:henkan-sn}) defines a double covering map $pr:M\to X$. 
  It is easy to see that the image of 
  $pr^{*}: H^1(X;\CL_0) \to H^1(M;\CL)$
  coincides with $H^{(1)}$. 
  Indeed, $[d\log \frac{1-t}{t}]$ and $[d\log \frac{1-z t}{t}]$, which form a basis of $H^1(X;\CL_0)$, 
  are mapped to $2[\phi_3]$ and $2[\phi_4]$, respectively.   
\end{Rem}

\subsection{Proof of (\ref{eq:C22-2})}\label{subsec:proof-of-C22}
In this subsection, we prove the equality (\ref{eq:C22-2}). 
Recall that we use the substitution $\al=a+\frac{1}{2}$, $\be=b-\frac{1}{2}$, $\ga=c$. 
We set 
\begin{align*}
  C'_{22}&:=2a(a+1)\cdot C_{22} \\
  &=\frac{1}{2\pi^2 \vth_{3}(0)^4 }
    \Big(-(2a+1) \frac{\vth_{1}'''(0)}{\vth_{1}'(0)}+(2a-2c+1)\frac{\vth''_{2}(0)}{\vth_{2}(0)}
    +(2b-1)\frac{\vth''_{3}(0)}{\vth_{3}(0)} +(4a-2b+2c+3)\frac{\vth''_{4}(0)}{\vth_{4}(0)}\Big)  .
\end{align*}
It suffices to show the equality $C'_{22}=(a-b+1)\la(\tau)+c$. 

In the proof, we use the formulas in \cite{WW}. 
We note that the definitions of our theta functions are
slightly different from those defined in \cite[\S 21.11]{WW}.
If we write $\vth_i^{\rm W}$ for the theta function defined in \cite{WW},
then we have $\vth_i(u;\tau)=\vth_i^{\rm W}(\pi u;e^{\pii \tau})$. 
By \cite[\S 21.41]{WW}, we have
\begin{align}
  \label{eq:theta1-0-diff}
  \frac{\vth_1'''(0)}{\vth_1'(0)}
  &=\pi^2 \left( -1+24\sum_{n=1}^{\infty}\frac{q^{n}}{(1-q^{n})^2} \right)
    =\frac{\vth_2''(0)}{\vth_2(0)}+\frac{\vth_3''(0)}{\vth_3(0)}+\frac{\vth_4''(0)}{\vth_4(0)}, \\
  \label{eq:theta2-0-diff}
  \frac{\vth_2''(0)}{\vth_2(0)}
  &=\pi^2 \left( -1-8\sum_{n=1}^{\infty}\frac{q^{n}}{(1+q^{n})^2}  \right), \\
  \label{eq:theta3-0-diff}
  \frac{\vth_3''(0)}{\vth_3(0)}
  &=-8\pi^2 \sum_{n=1}^{\infty}\frac{q^{n-\frac{1}{2}}}{\big( 1+q^{n-\frac{1}{2}} \big)^2} , \\
  \label{eq:theta4-0-diff}
  \frac{\vth_4''(0)}{\vth_4(0)}
  &=8\pi^2 \sum_{n=1}^{\infty}\frac{q^{n-\frac{1}{2}}}{\big( 1-q^{n-\frac{1}{2}} \big)^2} ,
\end{align}
where $q=\bfe (\tau)$.
Using the Eisenstein series of weight $2$:
\begin{align*}
  G_{2}(\tau )
  :=\frac{\pi^2}{3}-8\pi^{2}
  \sum_{n=1}^{\infty}\frac{nq^{n}}{1-q^{n}}
  =\frac{\pi^2}{3}\left( 1-24
  \sum_{n=1}^{\infty}\frac{q^{n}}{(1-q^{n})^2} \right) ,
\end{align*}
we rewrite the formulas (\ref{eq:theta1-0-diff})--(\ref{eq:theta4-0-diff}) as 
\begin{align}
  \label{eq:theta1-0-diff-2}
  \frac{\vth_1'''(0)}{\vth_1'(0)}
  &=-3 G_2(\tau)
    =\pi^2 \left( -1+24
    \sum_{n=1}^{\infty}\frac{nq^{n}}{1-q^{n}} \right), \\
  \label{eq:theta2-0-diff-2}
  \frac{\vth_2''(0)}{\vth_2(0)}
  &=-4G_2 (2\tau) +G_2(\tau)
    =\pi^2 \left( -1+8\sum_{n=1}^{\infty}\frac{(-1)^n nq^{n}}{1-q^{n}} \right), \\
  \label{eq:theta3-0-diff-2}
  \frac{\vth_3''(0)}{\vth_3(0)}
  &=4G_2 (2\tau) -5G_2 (\tau) +G_2 \Big( \frac{\tau}{2} \Big)
    =-8\pi^2 \sum_{n=1}^{\infty}\frac{(-1)^n nq^{\frac{n}{2}}}{1-q^{n}} , \\
  \label{eq:theta4-0-diff-2}
  \frac{\vth_4''(0)}{\vth_4(0)}
  &=G_2 (\tau) -G_2 \Big( \frac{\tau}{2} \Big)
    =8\pi^2 \sum_{n=1}^{\infty}\frac{nq^{\frac{n}{2}}}{1-q^{n}} .
\end{align}
Therefore, $C'_{22}$ is reduced to
\begin{align}
  \nonumber
  \frac{1}{\pi^2 \vth_{3}(0)^4 }
  \Big[
  &2(a-b+c+1)\big( 2G_{2}(2\tau )- G_{2}(\tau ) \big) \\
  \label{eq:LHS-by-E2}
  & \quad 
     -2(a-b+1)\Big(4G_{2}(2\tau )+G_{2}\Big(\frac{\tau }{2}\Big)-4G_{2}(\tau )\Big)
     +c\Big( 2G_{2}(\tau )-G_{2}\Big(\frac{\tau }{2}\Big)\Big)\Big] .
\end{align}
We consider Jacobian elliptic functions 
\begin{align*}
\sn(2Ku)
   :=
   \frac{\vth_{3}(0)}{\vth_{2}(0)}\frac{\vth_{1}(u)}{\vth_{4}(u)},  \quad 
\cn(2Ku)
   :=
   \frac{\vth_{4}(0)}{\vth_{2}(0)}\frac{\vth_{2}(u)}{\vth_{4}(u)},  \quad
\dn(2Ku)
   :=
   \frac{\vth_{4}(0)}{\vth_{3}(0)}\frac{\vth_{3}(u)}{\vth_{4}(u)} ,
\end{align*}
and 
\begin{align*}
   \cs(2Ku)
   :=
   \frac{\cn(2Ku)}{\sn(2Ku)}, \quad
   \ds(2Ku)
   :=
   \frac{\dn(2Ku)}{\sn(2Ku)}, \quad
   \ns(2Ku)
   :=
   \frac{1}{\sn(2Ku)},
\end{align*}
where $K=\pi \vth_{3}(0)^{2}/ 2$.
The elliptic functions $\cs$, $\ds$ and $\ns$ have the following series
\cite[\S 22.61]{WW}:
\begin{align}
  \nonumber
  2K\cs(2Ku)
  &=\pi \cot (\pi u) -4\pi \sum_{n=1}^{\infty} \frac{q^{n} \sin (2n\pi u)}{1+q^{n}} \\
  \label{eq:cs-Laurent-1}
  &=\frac{1}{u}-\frac{\pi^2}{3} \left( 1+24 \sum_{n=1}^{\infty} \frac{nq^{n}}{1+q^{n}} \right) u+\cdots
    , \\
  \nonumber
  2K\ds(2Ku)
  &=\frac{\pi}{\sin (\pi u)} -4\pi
    \sum_{n=1}^{\infty} \frac{q^{n-\frac{1}{2}} \sin ((2n-1)\pi u)}{1+q^{n-\frac{1}{2}}} \\
  \label{eq:ds-Laurent-1}
  &=\frac{1}{u} +\frac{\pi^2}{6}\left( 1-24 \sum_{n=1}^{\infty} \frac{(2n-1)q^{n-\frac{1}{2}}}{1+q^{n-\frac{1}{2}}} \right) u +\cdots
    , \\
  \nonumber
  2K\ns(2Ku)
  &=\frac{\pi}{\sin (\pi u)} +4\pi
    \sum_{n=1}^{\infty} \frac{q^{n-\frac{1}{2}} \sin ((2n-1)\pi u)}{1-q^{n-\frac{1}{2}}} \\
  \label{eq:ns-Laurent-1}
  &=\frac{1}{u} +\frac{\pi^2}{6}\left( 1+24 \sum_{n=1}^{\infty} \frac{(2n-1)q^{n-\frac{1}{2}}}{1-q^{n-\frac{1}{2}}} \right) u +\cdots
    .
\end{align}
From Taylor expansions of $\cn$, $\sn$ and $\dn$ \cite[\S 22.34]{WW}:
\begin{align*}
\sn(2Ku)
   &=
      2Ku
      - \frac{1+\lambda(\tau)}{6} (2K)^{3} u^3
      +\cdots,  \\
\cn(2Ku)
   &=
      1
      -\frac{1}{2}(2K)^2 u^2
      +\cdots, \nonumber \\
\dn(2Ku)
   &=
      1
      - \frac{\lambda (\tau)}{2} (2K)^{2} u^{2}
      + \cdots,
\end{align*}
we have Laurent expansions of $\cs$, $\ds$ and $\ns$:
\begin{align}
  \label{eq:cs-Laurent-2}
  2K \cs(2Ku)
   &=
     \frac{1}{u}
     +\left(-\frac{1}{3}+\frac{\lambda (\tau)}{6} \right) (2K)^2 u
     +\cdots,  \\
  \label{eq:ds-Laurent-2}
  2K\ds(2Ku)
   &=
     \frac{1}{u}
     +\left(\frac{1}{6}-\frac{\lambda(\tau)}{3} \right) (2K)^2 u
     +\cdots,  \\
  \label{eq:ns-Laurent-2}
  2K\ns(2Ku)
   &=
     \frac{1}{u}
     +\left(\frac{1}{6}+\frac{\lambda(\tau)}{6} \right) (2K)^2 u
     +\cdots .
\end{align}
Comparing the coefficients of $u$ in
(\ref{eq:cs-Laurent-1})--(\ref{eq:ns-Laurent-1}) and 
(\ref{eq:cs-Laurent-2})--(\ref{eq:ns-Laurent-2}), respectively, we have
\begin{align*}
  1+24 \sum_{n=1}^{\infty} \frac{nq^{n}}{1+q^{n}}
  &=\left(1-\frac{\lambda(\tau)}{2} \right) \vth_3(0)^4,\\
  1-24 \sum_{n=1}^{\infty} \frac{(2n-1)q^{n-\frac{1}{2}}}{1+q^{n-\frac{1}{2}}}
  &=\left(1-2\lambda(\tau) \right) \vth_3(0)^4,\\
  1+24 \sum_{n=1}^{\infty} \frac{(2n-1)q^{n-\frac{1}{2}}}{1-q^{n-\frac{1}{2}}}
  &=\left(1+\lambda(\tau) \right) \vth_3(0)^4 .
\end{align*}
A simple calculation shows that 
\begin{align*}
  &2G_{2}(2\tau )-G_{2}(\tau )
    =\frac{\pi^2}{3} \left(1-\frac{\lambda(\tau)}{2} \right) \vth_3(0)^4, \\
  &4G_{2}(2\tau )+G_{2}\Big(\frac{\tau }{2} \Big)-4G_{2}(\tau )
    =\frac{\pi^2}{3} (1-2\la (\tau)) \vth_3(0)^4, \\
  &2G_{2}(\tau )-G_{2}\Big(\frac{\tau }{2}\Big)
    =\frac{\pi^2}{3} \left(1+\lambda(\tau) \right) \vth_3(0)^4.  
\end{align*}
Substituting the above formulas into (\ref{eq:LHS-by-E2}), we obtain
\begin{align*}
  C'_{22}&=\frac{1}{3}\Big[
  2(a-b+c+1)\left(1-\frac{\lambda(\tau)}{2} \right)
  -2(a-b+1)(1-2\la (\tau)) 
    +c\left(1+\lambda(\tau) \right)\Big] \\
  &=(a-b+1)\la(\tau)+c,
\end{align*}
which completes the proof. 

\section{Structure of the twisted homology group}\label{sec:homology}
\subsection{Intersection matrix and period integrals for $\{ [\si_j] \}$}
The homology intersection numbers are given in \cite{Ghazouani-Pirio} and \cite{G-RW-intersection}. 
Applying (\ref{eq:RW-setting}) to the result in \cite{G-RW-intersection}, we immediately obtain 
the following formula. 
\begin{Fact}[{\cite[Proposition 3.4.1]{Ghazouani-Pirio}, \cite[Corollary 3.3]{G-RW-intersection}}]
  The intersection matrix $H' =(\langle [\si_i],[\si_j^{\vee}]\rangle_{\h})_{i,j=1,\dots ,4}$  
  is given by
  \begin{align*}
    H'=
    \begin{pmatrix}
      \frac{1-\bfe(c_1 +c_2)}{(1-\bfe(c_1))(1-\bfe(c_2))}&-\frac{1}{1-\bfe(c_2)} 
      &0&\frac{\bfe(c_1)(1-\bfe(-c_0))}{1-\bfe(c_1)} \\
      -\frac{\bfe(c_2)}{1-\bfe(c_2)}&\frac{1-\bfe(c_2 +c_3)}{(1-\bfe(c_2))(1-\bfe(c_3))}
      &-\frac{1}{1-\bfe(c_3)}&0 \\
      0&-\frac{\bfe(c_3)}{1-\bfe(c_3)}
      &\frac{1-\bfe(c_3 +c_4)}{(1-\bfe(c_3))(1-\bfe(c_4))}&0 \\
      \frac{1-\bfe(c_0)}{1-\bfe(c_1)}&0
      &0&\frac{(1-\bfe(c_0))(\bfe(c_0)-\bfe(c_1))}{\bfe(c_0)(1-\bfe(c_1))}
    \end{pmatrix}.
  \end{align*}
\end{Fact}

The period pairing between $\{ [\phi_j] \}_{j=1,\dots ,4}$ and $\{ [\si_j] \}_{j=1,\dots ,4}$ 
can be expressed in terms of ${}_2F_1$. 
\begin{Lem}\label{lem:W-integral-hypergeom}
  We have
  \begin{align}
    \label{eq:integral-phi3-1}
    \int_{\si_1}T(u)\phi_3 
    &=\frac{\Ga(\al)\Ga(\ga-\al)}{2\Ga(\ga)}
      \vth_{2}(0)^{2\ga} \vth_{3}(0)^{-2\al -2\be} \vth_{4}(0)^{-2\ga+2\al+2\be}
      {}_2F_1 (\al ,\be ,\ga ;\la(\tau)) , \\
    \nonumber
    \int_{\si_3}T(u)\phi_3  
    &=-\bfe \Big(\frac{1}{2} (\al +\be-\ga)\Big) \cdot 
      \frac{\Ga(1-\be)\Ga(1-\ga+\be)}{2\Ga(2-\ga)}
       \\
    \label{eq:integral-phi3-3}
    &\qquad  
      \cdot \vth_{2}(0)^{4-2\ga} \vth_{3}(0)^{2\al +2\be-4} \vth_{4}(0)^{2\ga-2\al-2\be}
      {}_2F_1 (1-\be ,1-\al ,2-\ga ;\la(\tau)),\\
    \label{eq:integral-phi3-4}
    \int_{\si_4} T(u)\phi_3 
    &=(1-\bfe(\ga-\al)) \int_{\si_1} T(u)\phi_3 , \\
    \nonumber
    \int_{\si_2}T(u)\phi_3 
    &=-\frac{1}{\bfe(2\al -2\ga)(1-\bfe(\ga))}
      \Bigg( (1-\bfe(\al))\int_{\si_1}T(u)\phi_3 \\
    \label{eq:integral-phi3-2}
    &\qquad \qquad 
      +\bfe(2\al+2\be-2\ga)(1-\bfe(\ga-\be))\int_{\si_3}T(u)\phi_3 \Bigg) .
  \end{align}
  We can obtain formulas for the integrals of $\phi_1$, $\phi_2$ and $\phi_4$ 
  by replacing $\al$, $\be$ and $\ga$ according to the following rule. 
  \begin{align*}
    \begin{array}{c|ccc}
      &\al & \be & \ga \\ \hline
      \phi_1 & \al+\frac{1}{2} & \be+\frac{1}{2} & \ga+1 \\
      \phi_2 & \al-\frac{1}{2} & \be+\frac{1}{2} & \ga \\
      \phi_4 & \al & \be+1 & \ga+1 
    \end{array}
  \end{align*}
\end{Lem}
\begin{proof}
  The equality (\ref{eq:integral-phi3-1}) is the same as (\ref{eq:W-integral-gauss}). 
  Integration by substitution shows 
  the equalities (\ref{eq:integral-phi3-3}) and (\ref{eq:integral-phi3-4}). 
  To show (\ref{eq:integral-phi3-2}), it suffices to consider the boundary of 
  a (locally finite) $2$-chain defined by a parallelogram surrounded by the path
  $(0 \to \frac{1}{2} \to \frac{1+\tau}{2} \to \frac{\tau}{2} \to \frac{-1+\tau}{2} \to -\frac{1}{2} \to 0)$. 
\end{proof}

\subsection{Involution and decomposition}
The involution $\iota$ introduced in \S\ref{subsec:involution-cohomology} 
naturally derives an involution on the twisted homology groups. 
By the compatibility among the period pairings and the intersection forms, 
we also obtain eigenspace decompositions of the twisted homology groups. 
\begin{Cor}\label{cor:homology-decomp}
  The twisted homology groups $H_1(M;\CL^{\vee})$ and $H_1(M;\CL)$ decompose as 
  \begin{align*}
    H_1(M;\CL^{\vee})=H_{(-1)}^{\vee} \op H_{(1)}^{\vee} ,\quad  
    H_1(M;\CL)=H_{(-1)}\op H_{(1)}.  
  \end{align*}
  These direct summands are characterized by the following properties: 
  \begin{align*}
    H_{(\pm 1)}^{\vee} &=\left\{ [\si]\in H_1(M;\CL^{\vee}) \, \left| \,
    \int_{\si} T(u) \vph =0 \ \text{ for all } \ [\vph] \in H^{(\mp 1)} \right. \right\}, \\
    H_{(\pm 1)} &=\left\{ [\si]\in H_1(M;\CL) \, \left| \,
    \int_{\si} T(u)^{-1} \vph =0 \ \text{ for all } \ [\vph] \in H^{(\mp 1)\vee} \right. \right\}. 
  \end{align*}
  Further, $H_{(\pm 1)}^{\vee}$ and $H_{(\mp 1)}$ are orthogonal to each other with respect to 
  the intersection form $\langle \bu ,\bu \rangle_{\h}$. 
\end{Cor}
\begin{proof}
  By the isomorphisms $\mathsf{PD}:H_1(M;\CL)\overset{\sim}{\to} H^1(M;\CL)$ and 
  $\mathsf{PD}^{\vee}:H_1(M;\CL^{\vee})\overset{\sim}{\to} H^1(M;\CL^{\vee})$ 
  of Poincar\'e duality, we set 
  $H_{(\pm 1)}=\mathsf{PD}^{-1}(H^{(\pm 1)})$ and $H_{(\pm 1)}^{\vee}=(\mathsf{PD}^{\vee})^{-1}(H^{(\pm 1)\vee})$. 
  We thus obtain the decompositions. 
  By the compatibility between the period pairings and the intersection forms (cf. \cite{KY}), we have
  \begin{align*}
    &\langle [\si],[\si'] \rangle_{\h} =-\langle \mathsf{PD}([\si']),\mathsf{PD}^{\vee}([\si])\rangle_{\ch} ,\\
    &\langle [\vph ],\mathsf{PD}^{\vee}([\si])\rangle_{\ch} =-\int_{\si} T(u) \vph ,\quad 
    \langle \mathsf{PD}([\si']),[\vph'] \rangle_{\ch} =\int_{\si'} T(u)^{-1} \vph' 
  \end{align*}
  for $[\si]\in H_1(M;\CL^{\vee})$, $[\si']\in H_1(M;\CL)$, 
  $[\vph] \in H^1(M;\CL)$ and $[\vph'] \in H^1(M;\CL^{\vee})$, 
  which implies the properties of 
  $H_{(\pm 1)}^{\vee}$ and $H_{(\pm 1)}$. 
\end{proof}

\subsection{Eigenvectors with respect to the involution}
The twisted cycles $[\si_1],\dots ,[\si_4]$ do not belong to $H_{(\pm 1)}^{\vee}$. 
In this subsection, we construct a basis of $H_{(\pm 1)}^{\vee}$. 
\begin{Th}\label{th:homology-eigen}
  We set 
  \begin{align*}
    \si_{1\pm}
    &=\mp \frac{1}{2\bfe(\ga-\al)} \big(  -(1 \pm \bfe(\ga-\al)) \si_{1} +\si_{4}\big) ,\\
    \si_{2\pm}
    &=\pm \frac{1}{2\bfe(\be+\ga)}\Big(
      \frac{1 - \bfe(2\al -\ga)}{\bfe(2\al -2\ga)}\si_{1} +(1-\bfe(\ga))\si_{2} 
      +\bfe(2\be)(1 \pm \bfe(\ga-\be))\si_{3} +\bfe(\ga) \si_{4}
  \Big). 
  \end{align*}
  Then, $[\si_{1\pm}],[\si_{2\pm}] $ form a basis of $H_{(\pm 1)}^{\vee}$ that satisfy
  \begin{align*}
    \int_{\si_{1\pm}} T(u) \vph_{\pm} =\int_{\si_1} T(u) \vph_{\pm} ,\quad 
    \int_{\si_{2\pm}} T(u) \vph_{\pm} =\int_{\si_3} T(u) \vph_{\pm}
  \end{align*}
  for any $[\vph_{\pm}] \in H^{(\pm 1)}$. 
\end{Th}
Recall that $[\phi_1]$ and $[\phi_2]$ (resp. $[\phi_3]$ and $[\phi_4]$) form a basis of 
$H^{(-1)}$ (resp. $H^{(1)}$). 
If we take $\vph_{+}\in \{ \phi_3,\phi_4 \}$ and $\vph_{-}\in \{ \phi_1,\phi_2 \}$, 
then these integrals are expressed in terms of ${}_2F_1$ 
by Lemma~\ref{lem:W-integral-hypergeom}. 
\begin{proof}
  Corollary~\ref{cor:homology-decomp} and 
  the equalities (\ref{eq:integral-phi3-4}), (\ref{eq:integral-phi3-2}) in Lemma~\ref{lem:W-integral-hypergeom} 
  imply that  
  the twisted cycles 
  \begin{align*}
    \si'_{1\pm}&=\si_{4} -(1 \pm \bfe(\ga-\al)) \si_{1} , \\
    \si'_{2\pm}&=\si_{2} +\frac{1}{\bfe(2\al -2\ga)(1-\bfe(\ga))}
            \Big( (1 \pm \bfe(\al))\si_{1}
            +\bfe(2\al+2\be-2\ga)(1 \pm \bfe(\ga-\be))\si_{3} \Big)
  \end{align*}
  form a basis of $H_{(\pm 1)}^{\vee}$. 
  For any $[\vph_{\pm}] \in H^{(\pm 1)}$, a direct computation shows that
  \begin{align*}
    \int_{\si'_{1\pm}} T(u)\vph_{\pm}
    &=\mp 2\bfe(\ga-\al)\int_{\si_1}T(u)\vph_{\pm} ,  \\
    \int_{\si'_{2\pm}} T(u)\vph_{\pm}
    &=\frac{\pm 2}{\bfe(\al -2\ga)(1- \bfe(\ga))}
      \Big( \int_{\si_{1}}T(u)\vph_{\pm} 
      + \bfe(\al+\be-\ga)\int_{\si_{3}}T(u)\vph_{\pm} \Big) .
  \end{align*}
  Therefore, the twisted cycles 
  \begin{align*}
    \si_{1\pm}
    =\mp \frac{1}{2\bfe(\ga-\al)} \si'_{1\pm} ,\quad
    \si_{2\pm}
    =\pm \frac{1-\bfe(\ga)}{2\bfe(\be+\ga)}\Big(
    \si'_{2\pm}+\frac{\bfe(\ga)}{1- \bfe(\ga)} \si'_{1\pm}
    \Big)     
  \end{align*}
  satisfy the required properties. 
\end{proof}

\section{Twisted period relations}\label{sec:TPR}
Similarly to the argument in \cite{CM}, we have the following quadratic relations among 
the Wirtinger integrals, which is called the twisted period relations. 

We use the twisted cycles $(\si_{1-},\si_{2-},\si_{1+},\si_{2+})$ and 
cocycles $(\phi_1,\phi_2,\phi_3,\phi_4)$ as bases of the twisted homology and cohomology groups, 
respectively. 
We set 
\begin{align*}
  P_{+}=
  \begin{pmatrix}
    \int_{\si_{1-}} T(u) \phi_1 & \cdots & \int_{\si_{2+}} T(u) \phi_1 \\
    \vdots &  & \vdots \\
    \int_{\si_{1-}} T(u) \phi_4 & \cdots & \int_{\si_{2+}} T(u) \phi_4
  \end{pmatrix}
  ,\quad 
  P_{-}=  
  \begin{pmatrix}
    \int_{\si_{1-}^{\vee}} T(u)^{-1} \phi_1 & \cdots & \int_{\si_{2+}^{\vee}} T(u)^{-1} \phi_1 \\
    \vdots &  & \vdots \\
    \int_{\si_{1-}^{\vee}} T(u)^{-1} \phi_4 & \cdots & \int_{\si_{2+}^{\vee}} T(u)^{-1} \phi_4
  \end{pmatrix} ,
\end{align*}
which we call the period matrices. 
By the above discussion, the homology intersection matrix $H$, the cohomology intersection matrix $C$, and 
the period matrices $P_{+}$, $P_{-}$ become block diagonal: 
\begin{align*}
  H
  &=\begin{pmatrix}
      H(-1) & 0 \\ 0 & H(1) 
    \end{pmatrix},\  
  H(\pm 1) =\frac{1-\bfe(\ga)}{2}
    \begin{pmatrix}
      \frac{1}{(1 \mp \bfe(\ga-\al))(1 \mp \bfe(\al))} & 0 \\
      0 & -\frac{1}{(1 \mp \bfe(\ga-\be))(1 \mp \bfe(\be))}
    \end{pmatrix} ,\\
  C
  &=\begin{pmatrix}
      C(-1) & 0 \\ 0 & C(1) 
    \end{pmatrix},\  
  C(-1)=\tpi 
    \begin{pmatrix}
      0 & \frac{1}{2\al +1} \\
      \frac{1}{2\al -1} & C_{22}
    \end{pmatrix},\ 
  C(1)=\tpi
    \begin{pmatrix}
      \frac{2\ga}{2\al (2\ga -2\al)} &\frac{1}{2\al} \\
      \frac{1}{2\al} & \frac{2\be -2\al}{2\al \cdot 2\be}
    \end{pmatrix} ,\\
  P_{+}
  &=\begin{pmatrix}
      P_{+}(-1) & 0 \\ 0 & P_{+}(1)
    \end{pmatrix},\ 
  P_{+}(-1)=\Big( \int_{\si_j} T(u)\phi_i \Big)_{\substack{i=1,2 \\ j=1,3}} ,\ 
  P_{+}(1)=\Big( \int_{\si_j} T(u)\phi_i \Big)_{\substack{i=3,4 \\ j=1,3}} . 
\end{align*}
Note that $P_{-}$ is obtained by replacing $(\al,\be,\ga)$ with $(-\al,-\be,-\ga)$ in $P_{+}$. 
By Lemma~\ref{lem:W-integral-hypergeom} and Theorem~\ref{th:homology-eigen}, 
each entry of $P_{\pm}(\pm 1)$ can be expressed in terms of ${}_2F_1$. 

The twisted period relation $C =P_{+} \cdot \tp H^{-1} \cdot \tp P_{-}$ is decomposed into 
two relations among $2\times 2$ matrices: 
\begin{align}
  \label{eq:TPR-2x2}
  C(\pm 1) =P_{+}(\pm 1) \cdot \tp H(\pm 1)^{-1} \cdot \tp P_{-}(\pm 1) .
\end{align}
Similarly to \cite[Remark 2]{Watanabe-wirtinger-diff-eq}, these relations are essentially
the same as the twisted period relations for ${}_2 F_1$ obtained in \cite{CM}. 
For example, we obtain the following quadratic relation. 
Recall that the intersection number $C_{22}$ has a complicated expression (\ref{eq:C22-1}) by theta functions and 
it is reduced to a simple form (\ref{eq:C22-2}) by using $\la (\tau)$. 
\begin{Cor}
  The $(2,2)$-entry of the relation for $C(-1)$ in (\ref{eq:TPR-2x2}) is reduced to 
  \begin{align}
    \nonumber
    &(a-b+1)\la(\tau)+c \\
    \nonumber
    &=c\cdot 
      {}_2F_1 (a ,b ,c ;\la(\tau))
      {}_2F_1 (-a-1 ,-b+1 ,-c ;\la(\tau)) \\
    \label{eq:TPR-(2,2)}
    &\qquad 
      +\frac{a(a +1)(c-b)(c-b+1)}{c (1+c)(1-c)}
      \cdot \la(\tau)^2 \cdot
      {}_2F_1 (a+2 ,b ,2+c ;\la(\tau))
      {}_2F_1 (-a+1 ,-b+1 ,2-c ;\la(\tau)) ,
  \end{align}
  by setting $\al=a+\frac{1}{2}$, $\be=b-\frac{1}{2}$, $\ga=c$. 
\end{Cor}

\appendix
\section{Direct proof of (\ref{eq:TPR-(2,2)})}\label{sec:direct-proof}
In fact, we can prove (\ref{eq:TPR-(2,2)}) in terms of twisted (co)homology theory for 
the Euler-type integral representation (\ref{eq:gauss-integral}). 
We consider the twisted cohomology group $H^1(X;\CL_0)$ and homology group $H^1(X;\CL_0^{\vee})$
associated with $U(t)=t^{a} (1-t)^{c-a} (1-zt)^{-b}$, 
see Remark~\ref{rem:double-cover}. 
We set 
\begin{align*}
  p_{1+}&=\int_0^1 U(t)\frac{dt}{t(1-t)} ,\quad 
  p_{2+}=\int_{\frac{1}{z}}^{\infty} U(t) \frac{dt}{t(1-t)} , \\
  p_{1-}&=\int_0^1 U(t)^{-1} \frac{dt}{t^2(1-zt)} ,\quad 
  p_{2-}=\int_{\frac{1}{z}}^{\infty} U(t)^{-1} \frac{dt}{t^2(1-zt)}.
\end{align*}
Then, the twisted period relations for ${}_2F_1$ (e.g., \cite{CM}) imply 
\begin{align*}
  &\left\langle  \frac{dt}{t(1-t)} , \frac{dt}{t^2(1-zt)}\right\rangle_{\ch,0}  \\
  &=\frac{(1-\bfe(a))(1-\bfe(c-a))}{1-\bfe(c)} \cdot p_{1+}p_{1-}
  +\frac{(1-\bfe(-b))(1-\bfe(b-c))}{1-\bfe(-c)}\cdot p_{2+}p_{2-} \\
  &=\tpi \frac{c}{a(a+1)} \cdot {}_2F_1 (a,b,c;z) \cdot {}_2F_1 (-a-1,-b+1,-c;z) \\
  &\quad 
    +\tpi \frac{(c-b)(c-b+1)}{c(1+c)(1-c)} \cdot z^2 
    \cdot {}_2F_1 (a+2,b,2+c;z) \cdot {}_2F_1 (-a+1,-b+1,2-c;z) ,
\end{align*}
where $\langle \bu,\bu \rangle_{\ch,0} :H^1(X;\CL_0)\times H^1(X;\CL_0^{\vee}) \to \C$ is the cohomology intersection form. 
On the other hand, we can compute the cohomology intersection number in the same way as in \cite{M-k-form}: 
\begin{align*}
  \left\langle  \frac{dt}{t(1-t)} , \frac{dt}{t^2(1-zt)}\right\rangle_{\ch,0} 
  =\frac{\tpi}{a(a+1)} \cdot \big( (a-b+1)z+c \big).
\end{align*}
Therefore, the right-hand side of (\ref{eq:TPR-(2,2)}) equals $(a-b+1)\la(\tau)+c$. 

\begin{Rem}
  The assertion that
  the right-hand side of (\ref{eq:TPR-(2,2)}) equals $(a-b+1)\la(\tau)+c$
  is also proved by using the power series expansion in $\la (\tau)$ and 
  Whipple's transformation formula \cite[Theorem 3.3.3]{AAR}:
  \begin{align*}
    {}_4F_3 \left(
    \begin{matrix}
      -n,a,b,c \\ d,e,f
    \end{matrix}
    ; 1 \right)
    =\frac{(e-a)_n (f-a)_n}{(e)_n (f)_n} \cdot 
    {}_4F_3 \left(
    \begin{matrix}
      -n,a,d-b,d-c \\ d,a-e-n+1,a-f-n+1
    \end{matrix}
    ; 1 \right) ,
  \end{align*}
  where $n$ is a non-negative integer and $a+b+c-n+1=d+e+f$.
  The coefficient of $\la(\tau)^n$ in
  the first term on the right-hand side of (\ref{eq:TPR-(2,2)}) is
  \begin{align}
    \nonumber
    &c\cdot \frac{(-a-1)_n (-b+1)_n}{(-c)_n n!} \cdot 
    {}_4F_3 \left(
    \begin{matrix}
      -n,b,a,1-n+c \\ 2-n+a,c,-n+b
    \end{matrix}
    ; 1 \right) \\
    \label{eq:4F3-1}
    &=c\cdot \frac{(-a-1)_n (-b+1)_n}{(-c)_n n!} 
      \cdot \frac{(c-b)_n (-n)_n}{(c)_n (-n+b)_n} \cdot 
    {}_4F_3 \left(
    \begin{matrix}
      -n,b,-n+2,1+a-c \\ 2-n+a,b-c-n+1,1
    \end{matrix}
    ; 1 \right) .
  \end{align}
  In particular, the cases of $n=0$ and $n=1$ are equal to
  $c$ and $a-b+1$, respectively.
  On the other hand, 
  the coefficient of $\la(\tau)^n$ ($n \geq 2$) in
  the second term on the right-hand side of (\ref{eq:TPR-(2,2)}) is
  \begin{align}
    \nonumber
    &\frac{a(a +1)(c-b)(c-b+1)}{c (1+c)(1-c)}
      \cdot \frac{(-a+1)_{n-2} (-b+1)_{n-2}}{(2-c)_{n-2} (n-2)!} \cdot 
    {}_4F_3 \left(
    \begin{matrix}
      -n+2,b,a+2,1-n+c \\ 2-n+a,2+c,-n+b
    \end{matrix}
    ; 1 \right) \\
    \label{eq:4F3-2}
    &=(-c)\cdot \frac{(-a-1)_n (-b+1)_{n-2}}{(-c)_n (n-2)!} 
      \cdot \frac{(c-b)_n (2-n)_{n-2}}{(c)_n (2-n+b)_{n-2}} \cdot 
    {}_4F_3 \left(
    \begin{matrix}
      -n,b,-n+2,1+a-c \\ 2-n+a,b-c-n+1,1
    \end{matrix}
    ; 1 \right) . 
  \end{align}
  Then we have $\text{(\ref{eq:4F3-1})}+\text{(\ref{eq:4F3-2})}=0$ for $n\geq 2$.
\end{Rem}

\begin{Ack}
  The authors are grateful to Professor Akihito Ebisu for pointing out the idea 
  used in Appendix~\ref{sec:direct-proof}. 
  This work was supported by JSPS KAKENHI Grant Numbers JP20K14276 and JP24K06680.
\end{Ack}

\end{document}